\documentclass[12pt]{article}
\usepackage[left=1in,top=1in,right=1in,bottom=1in,letterpaper]{geometry}

\usepackage{amsmath}
\usepackage{amsfonts}
\usepackage{algorithm}
\usepackage{algorithmic}
   \newtheorem{proposition}{Proposition}
   \newtheorem{definition}{Definition}
   \newtheorem{theorem}{Theorem}
   \newtheorem{proof}{Proof}
   \newtheorem{lemma}{Lemma}
\begin{document}

\title{A note on alternating minimization algorithms: Bregman frame }
\author{Tao Sun\thanks{
College of Science, National University of Defense Technology,
Changsha, 410073, Hunan,  China. Email: \texttt{nudtsuntao@163.com} }
\and Lizhi Cheng$^{*}$\thanks{The State Key Laboratory for High Performance Computation, National University of Defense Technology,
Changsha, 410073, Hunan,  China. Email: \texttt{clzcheng@nudt.edu.cn}}
}

\maketitle

\begin{abstract}
In this paper, we propose a Bregman frame for several classical alternating minimization algorithms. In the frame,  these algorithms have   uniform mathematical formulation. We also present convergence analysis for the frame algorithm. Under the Kurdyka-{\L}ojasiewicz property,  stronger convergence is obtained.
\end{abstract}


\section{Introduction}
In this paper, we are devoted to solving the following  problem
\begin{equation}\label{model}
    \min_{y,z} \Phi(y,z)=f(y)+H(y,z)+g(z),
\end{equation}
where $f$, $g$ and $H$ are all closed, and the function $H$ is continuously differentiable over $\textrm{dom}(f)\times \textrm{dom}(g)$.
It is natural to consider using the alternating minimization method for (\ref{model}), i.e., minimizing only one variable and fixing other ones in each iteration. Mathematically, the \textbf{A}lternating \textbf{M}inimization (AM) can be presented as
\begin{eqnarray}
y^{k+1}&\in& \textrm{arg}\min_{y} H(y,z^k)+f(y),\nonumber\\
z^{k+1}&\in& \textrm{arg}\min_{z} H(y^{k+1},z)+g(z).\nonumber
\end{eqnarray}
 A. Beck first proves the convergence of  AM when both functions $f$ and $g$ are convex\cite{beck2015convergence}. The sublinear convergence rate is derived provided $\Phi$ is further convex\cite{beck2015convergence}. In each iteration, the AM algorithm need to solve two subproblems which may not have explicit solutions.   For practical perspective, paper \cite{bolte2014proximal} proposes the  \textbf{P}roximal \textbf{L}inearized \textbf{A}lternating \textbf{M}inimization (PLAM) algorithm which reads as
\begin{eqnarray}
y^{k+1}&\in& \textrm{arg}\min_{y} \langle\nabla_y H(y^k,z^k),y-y^k\rangle+\frac{\alpha_k}{2}\|y-y^k\|_2^2+f(y),\nonumber\\
z^{k+1}&\in& \textrm{arg}\min_{z} \langle\nabla_z H(y^{k+1},z^k),z-z^k\rangle+\frac{\beta_k}{2}\|z-z^k\|_2^2+g(z).\nonumber
\end{eqnarray}
If the proximal maps of $f$ and $g$ are easy to calculate, so are $y^{k+1}$ and $z^{k+1}$. The the global convergence of PLAM is proved under several reasonable assumptions \cite{bolte2014proximal}. The convergence results under convex case are also proved  \cite{shefi2015rate}. Besides AM and PLAM, the \textbf{A}ugmented  \textbf{A}lternating \textbf{M}inimization (AAM) is also considered by literature \cite{attouch2010proximal}; the scheme of AAM can be described as
 \begin{eqnarray}
y^{k+1}&\in& \textrm{arg}\min_{y} H(y,z^k)+f(y)+\frac{\alpha_k}{2}\|y-y^k\|_2^2,\nonumber\\
z^{k+1}&\in& \textrm{arg}\min_{z} H(y^{k+1},z)+g(z)+\frac{\beta_k}{2}\|z-z^k\|_2^2.\nonumber
\end{eqnarray}
The augmented terms in AAM is actually to derive a sufficient descend condition of the algorithm.  The convergence results of AAM are presented in \cite{attouch2010proximal}.

In applications, we may need to combine the algorithms above. The following problem, which is provided an example to illustrate the hybrid ideal, which reads as
\begin{equation}\label{problem2}
    \min_{y\in \mathbb{R}^{N_1}, z\in \mathbb{R}^{N_2}}\lambda_1\|y\|_1+\|Ay-z\|_2^2+\lambda_2\|z\|_{1,2},
\end{equation}
where $\lambda_1,\lambda_2>0$, and $\|\cdot\|_1$  and $\|\cdot\|_{1,2}$ are the $\ell_1$ norm and $\ell_{1,2}$ norm, respectively. When $z$ is fixed, minimizing the function above needs iterations; therefore, we use the PALM strategy. While when $y$ is fixed, the minimization  of the function is quite simple. This algorithm can be generalized as
 \begin{eqnarray}\label{AM-PLAM}
y^{k+1}&\in& \textrm{arg}\min_{y} H(y,z^k)+f(y),\nonumber\\
z^{k+1}&\in& \textrm{arg}\min_{z} \langle\nabla_z H(y^{k+1},z^k),z-z^k\rangle+\frac{\beta_k}{2}\|z-z^k\|_2^2+g(z).\nonumber
\end{eqnarray}
We call scheme (\ref{AM-PLAM}) as AM-PLAM. Under the hybrid  ideal, we can see that there are at least 8 classes of alternating minimization algorithms for two block case. As the the number of the blocks increases, more   alternating minimization algorithms  are generated.

In this note, we propose a framework for the algorithms mentioned above (including the hybrid ones). The frame is based on the Bregman distance. In the frame, all the algorithm enjoy the same mathematical formulation. Convergence results are also provided for the frame.

The rest of the paper is organized as follows. Section 2 presents the framework. Section 3 contains the convergence results. Section 4 concludes the paper.
\section{Bregman frame}
To propose the framework, we need recall the definition and properties of the Bregman distance.
The Bregman distance which is an extension of the squared Euclidean distance was proposed by \cite{bregman1967relaxation}. In recent years, it has been used for various models and algorithms in signal processing and machine learning research\cite{eckstein2014understanding,yin2010analysis,yin2008bregman}.  For a convex differential function $\phi$, the Bregman distance is defined as
\begin{equation}
    B_{\phi}(x,y):=\phi(x)-\phi(y)-\langle\nabla\phi(y),x-y\rangle.
\end{equation}
If $\phi(x)=\|x\|_2^2$, we have that $B_{\phi}(x,y)=\|x-y\|_2^2$. If $\phi=x^{\top}Mx$, where $M$ is a positive semidefinite matrix, then, $B_{\phi}(x,y)=\|M^{\frac{1}{2}}(x-y)\|_2^2=(x-y)^{\top}M(x-y)$.
\begin{proposition}
Let $\phi$ be a differentiable convex function and $B_{\phi}(x,y)$ is the Bregman distance. Then,

\textbf{1.} $B_{\phi}(x,y)\geq 0$, $B_{\phi}(x,x)=0$, for any $x,y\in dom(\phi)$.

\textbf{2.} For a fixed $y$, $B_{\phi}(x,y)$ is convex.

\textbf{3.} If $\phi$ is strongly convex with $\delta$, then, $B_{\phi}(x,y)\geq \frac{\delta}{2}\|x-y\|_2^2$, for any $x,y\in dom(\phi)$.
\end{proposition}
In this paper, the convex function $\phi$ and $\psi$ are assumed to have   Lipschitz  gradients with $L_{\phi},L_{\psi}>0$. And function $H(y,z)$ satisfies the following assumption
\begin{equation}
    \|\nabla_y H(\overline{y},z)-\nabla H(y,z)\|_2\leq L_1(\|\overline{y}-y\|_2)
\end{equation}
for any fixed $z\in  \textrm{dom}(g)$, and
\begin{equation}
    \|\nabla_z H(y,\overline{z})-\nabla H(y,z)\|_2\leq L_2(\|\overline{z}-z\|_2)
\end{equation}
for any fixed $y\in  \textrm{dom}(f)$.
The Bregman frame for the algorithm can be described as
\begin{eqnarray}\label{scheme}
y^{k+1}&\in& \textrm{arg}\min_{y} H(y,z^k)+f(y)+B_{\phi^k}(y,y^k),\nonumber\\
z^{k+1}&\in& \textrm{arg}\min_{z} H(y^{k+1},z)+g(z)+B_{\psi^k}(z,z^k),\nonumber
\end{eqnarray}
where $\{\phi^k,\psi^k\}_{k=0,1,2,\ldots}$ are differentiable functions.
\begin{algorithm}
\caption{Bregman Alternating Minimization Algorithm}
\begin{algorithmic}\label{alg1}
\REQUIRE   parameters $\alpha>0$, functions $\Phi,\Psi$\\
\textbf{Initialization}: $y^0,z^{0}$\\
\textbf{for}~$k=0,1,2,\ldots$ \\
~~~ $y^{k+1}\in \textrm{arg}\min_{y} H(y,z^k)+f(y)+B_{\phi^k}(y,y^k)$ \\
~~~ $z^{k+1}\in \textrm{arg}\min_{z} H(y^{k+1},z)+g(z)+B_{\psi^k}(z,z^k)$ \\
\textbf{end for}\\
\end{algorithmic}
\end{algorithm}
We explain that   AM, PLAM and AAM all can be regarded as the special case of  Algorithm 1.

\textbf{AM}: In Algorithm 1, we set $\phi^k=\psi^k\equiv 0$. Then, Algorithm 1 reduces to the AM algorithm.

\textbf{PLAM}: Let $\phi^k(y)=\frac{\alpha_k}{2}\|y\|_2^2-H(y,z^k)$, then, $\phi^k$ is differentiable.  If $\alpha_k>L_1$, we have that
\begin{eqnarray}
\langle \psi^k(\overline{y})-\phi^k(y),\overline{y}-y\rangle&=&\alpha_k\|\overline{y}-y\|_2^2-\langle \nabla_y H(\overline{y},z^k)-\nabla_y H(y,z^k) ,\overline{y}-y\rangle\nonumber\\
&\geq&(\alpha_k-L_1)\|\overline{y}-y\|_2^2\geq0.
\end{eqnarray}
Therefore, $\phi^k$ is convex. And it can be easily verified that calculating $y^{k+1}$ equals to
\begin{equation}
   \textrm{arg}\min_{y} H(y,z^k)+f(y)+B_{\phi^k}(y,y^k).
\end{equation}
Similarly, letting $\psi^k(z)=\frac{\beta_k}{2}\|z\|_2^2-H(y^{k+1},z)$, computing $z^{k+1}$ equals to
\begin{equation}
   \textrm{arg}\min_{z} H(y^{k+1},z)+g(y)+B_{\psi^k}(z,z^k).
\end{equation}

\textbf{AAM}: Letting $\phi^k(y)=\frac{\alpha_k}{2}\|y\|_2^2$ and $\psi^k(z)=\frac{\beta_k}{2}\|z\|_2^2$, Algorithm 1 reduces to the AAM algorithm.

We can easily see that the hybrid algorithms can also be generalized as the Bregman form. Thus,   Algorithm 1 provides a unified mathematical formulation for the alternating minimization algorithms.

\section{Convergence analysis}
\subsection{Preliminary tools}
Here, we collect some basic definitions and propositions in the following.
\begin{definition}[subdifferentials\cite{mordukhovich2006variational,rockafellar2009variational}] Let  $J: \mathbb{R}^N \rightarrow (-\infty, +\infty]$ be a proper and lower semicontinuous function.
\begin{enumerate}
  \item For a given $x\in dom (J)$, the Fr$\acute{e}$chet subdifferential of $J$ at $x$, written as $\hat{\partial}J (x)$, is the set of all vectors $u\in \mathbb{R}^N$ which satisfy
  $$\lim_{y\neq x}\inf_{y\rightarrow x}\frac{J(y)-J(x)-\langle u, y-x\rangle}{\|y-x\|_2}\geq 0.$$
When $x\notin dom (J)$, we set $\hat{\partial}J(x)=\emptyset$.

\item The (limiting) subdifferential, or simply the subdifferential, of $J$ at $x\in \mathbb{R}^N$, written as $\partial J(x)$, is defined through the following closure process
$$\partial J(x):=\{u\in\mathbb{R}^N: \exists x^k\rightarrow x, J(x^k)\rightarrow J(x)~\textrm{and}~ u^k\in \hat{\partial}J(x^k)\rightarrow u~\textrm{as}~k\rightarrow \infty\}.$$
\end{enumerate}
\end{definition}
It is easy to verify that the Fr$\acute{e}$chet subdifferential is convex and closed while the subdifferential is closed. When $J$ is convex,  the definition agrees with the one in convex analysis \cite{rockafellar2015convex} as
$$\partial J(x):=\{v: J(y)\geq J(x)+\langle v,y-x\rangle~~\textrm{for}~~\textrm{any}~~y\in \mathbb{R}^N\}.$$
Let $\{(x^k,v^k)\}_{k\in \mathbb{N}}$ be a sequence in $\mathbb{R}^N\times \mathbb{R}$ such that $(x^k,v^k)\in \textrm{graph }(\partial J)$. If $(x^k,v^k)$  converges to $(x, v)$ as $k\rightarrow +\infty$ and $J(x^k)$ converges to $v$ as $k\rightarrow +\infty$, then $(x, v)\in \textrm{graph }(\partial J)$. This indicates the following simple proposition.
\begin{proposition}\label{sublimit}
If $v^k\in \partial J(x^k)$, and $\lim_{k}v^k=v$ and $\lim_{k}x^k=x$. Then, we have that
\begin{equation}
    v\in \partial J(x).
\end{equation}
\end{proposition}
A necessary condition for $x\in\mathbb{R}^N$ to be a minimizer of $J(x)$ is
\begin{equation}\label{Fermat}
\textbf{0}\in \partial J(x).
\end{equation}
When $J$ is convex, (\ref{Fermat}) is also sufficient. A point that satisfies (\ref{Fermat}) is called (limiting) critical point. The set of critical points of $J(x)$ is denoted by $\textrm{crit}(J)$.
We call $J$ is strongly convex with $\nu$ if for any $x,y\in dom(J)$ and any $v\in \partial J(x)$, it holds that
$$J(y)\geq J(x)+\langle v,y-x\rangle+\frac{\nu}{2}\|y-x\|_2^2.$$

\begin{proposition}[\cite{nesterov2004introductory}]\label{sconvex}
Assume that $J(x)$ is strongly convex with $\nu$ and $x^*\in \textrm{arg}\min J(x)$. Then, for any $x\in dom(J)$, we have
\begin{equation}
    J(x)\geq J(x^*)+\frac{\nu}{2}\|x^*-x\|_2^2.
\end{equation}
\end{proposition}

\begin{proposition}\label{criL}
If $x^*=(y^*,z^*)$ is a critical point of $\Phi$, it must hold that
\begin{eqnarray}
 -\nabla H_{y}(y^*,z^*)&\in& \partial f(y^{*}),\nonumber\\
-\nabla H_{z}(y^*,z^*)&\in& \partial g(z^{*}).
\end{eqnarray}
\end{proposition}

\begin{definition}[\cite{bolte2014proximal,attouch2013convergence}]\label{KL}
(a) The function $J: \mathbb{R}^N \rightarrow (-\infty, +\infty]$ is said to have the  Kurdyka-{\L}ojasiewicz property at $\overline{x}\in dom(\partial J)$ if there
 exist $\eta\in (0, +\infty]$, a neighborhood $U$ of $\overline{x}$ and a continuous function $\varphi: [0, \eta)\rightarrow \mathbb{R}^+$ such that
\begin{enumerate}
  \item $\varphi(0)=0$.
  \item $\varphi$ is $C^1$ on $(0, \eta)$.
  \item for all $s\in(0, \eta)$, $\varphi^{'}(s)>0$.
  \item for all $x$ in $U\bigcap\{x|J(\overline{x})<J(x)<J(\overline{x})+\eta\}$, the Kurdyka-{\L}ojasiewicz inequality holds
\begin{equation}
  \varphi^{'}(J(x)-J(\overline{x}))\textrm{dist}(\textbf{0},\partial J(x))\geq 1.
\end{equation}
\end{enumerate}

(b) Proper lower semicontinuous functions which satisfy the Kurdyka-{\L}ojasiewicz inequality at each point of $dom(\partial J)$ are called KL functions.
\end{definition}
\subsection{Convergence results}
\begin{theorem}\label{gld}
For any $k$, we have that
\begin{equation}
    \Phi(y^{k+1},z^{k+1})\leq\Phi(y^{k+1},z^{k})\leq\Phi(y^{k},z^{k}).
\end{equation}
\end{theorem}
\begin{proof}
The scheme of  Algorithm 1 gives that
\begin{equation}\label{g1}
    H(y^{k+1},z^k)+f(y^{k+1})+B_{\phi^k}(y^{k+1},y^k)\leq H(y^{k},z^k)+f(y^k)+B_{\phi^k}(y^k,y^k).
\end{equation}
Note that $B_{\phi^k}(y^{k+1},y^k)\geq 0$ and  $B_{\phi^k}(y^{k},y^k)=0$, we can also have
\begin{equation}
    H(y^{k+1},z^k)+f(y^{k+1})\leq H(y^{k},z^k)+f(y^k).
\end{equation}
Then, we can obtain
\begin{equation}
    \Phi(y^{k+1},z^{k})\leq\Phi(y^{k},z^{k}).
\end{equation}
Similarly, we can derive the other inequality.
\end{proof}

\begin{lemma}\label{descend}
If the following condition holds
\begin{equation}\label{condition}
    \min\{\nu_{\phi^k},\nu_{\psi^k}\}>0,
\end{equation}
then, there exist $\rho>0$ such that
\begin{equation}
    \Phi(x^k)-\Phi(x^{k+1})\geq \rho\|x^k-x^{k+1}\|_2^2.
\end{equation}
\end{lemma}
\begin{proof}
Letting $\rho=\min\{\nu_{\phi^k},\nu_{\psi^k}\}$,
we easily see that
\begin{equation}
    B_{\phi^k}(y^{k+1},y^k)\geq\rho\|y^{k+1}-y^k\|_2^2
\end{equation}
Then, from (\ref{g1}), we can have that
\begin{equation}\label{l1t1}
    H(y^{k+1},z^k)+f(y^{k+1})+\rho\|y^{k+1}-y^k\|_2^2\leq H(y^{k},z^k)+f(y^k).
\end{equation}
Similarly, we can obtain that
\begin{equation}\label{l1t2}
    H(y^{k+1},z^{k+1})+g(z^{k+1})+\rho\|z^{k+1}-z^k\|_2^2\leq H(y^{k+1},z^{k})+g(z^{k}).
\end{equation}
Summing (\ref{l1t1}) and (\ref{l1t2}), we can obtain the result.
\end{proof}

\begin{lemma}\label{limit}
If the sequence $\{x^k=(y^k,z^k)\}_{k=0,1,2,\ldots}$ generated by Algorithm 1 is bounded and condition (\ref{condition})  holds and
\begin{equation}\label{condition2}
   \max\{L_{\phi^k},L_{\psi^k}\}<+\infty,
\end{equation}
 then we have
\begin{equation}
    \lim_{k}\|x^{k+1}-x^k\|_2=0.
\end{equation}
For any cluster point $x^*=(y^*,z^*)$, it is also a critical point of $\Phi$.
\end{lemma}
\begin{proof}
 The continuity of $\Phi$ indicates that $\{\Phi(x^k)\}_{k=0,1,2,\ldots}$ is bounded. From Lemma \ref{descend},  $\Phi(x^k)$ is decreasing. Thus, the sequence $\{\Phi(x^k)\}_{k=0,1,2,\ldots}$ is convergent, i.e., $\lim_{k}[\Phi(x^k)-\Phi(x^{k+1})]=0$.  With Lemma \ref{descend}, we have
\begin{equation}
    \lim_{k}\|z^{k+1}-z^k\|_2\leq\lim_{k}\frac{\Phi(x^k)-\Phi(x^{k+1})}{\rho}=0.
\end{equation}
For any cluster point $(y^*,z^*)$, there exists $\{k_j\}_{j=0,1,2,\ldots}$ such that $\lim_{j}(y^{k_j},z^{k_j})=(y^*,z^*)$. Then, we also have that $\lim_{j}(y^{k_j+1},z^{k_j+1})=(y^*,z^*)$.
From the scheme of Algorithm 1, we have the following conditions
\begin{eqnarray}
-\nabla_{y}H(y^{k_j+1},z^{k_j})+\nabla\phi^{k_j}(y^{k_j})-\nabla\phi^{k_j}(y^{k_j+1})&\in& \partial f(y^{k_j+1}),\nonumber\\
-\nabla_{z}H(y^{k_j+1},z^{k_j+1})+\nabla\psi^{k_j}(z^{k_j})-\nabla\psi^{k_j}(z^{k_j+1})&\in& \partial g(z^{k_j+1}).\nonumber
\end{eqnarray}
Letting $j\rightarrow+\infty$, with closedness of the subdifferential, we have that
\begin{eqnarray}
 -\nabla H_{y}(y^*,z^*)&\in& \partial f(y^{*}),\nonumber\\
-\nabla H_{z}(y^*,z^*)&\in& \partial g(z^{*}).\nonumber\\
\end{eqnarray}
From Proposition \ref{criL}, $x^*$ is a critical point of $\Phi$.
\end{proof}

\begin{lemma}\label{relative}
If condition (\ref{condition2}) holds, for any $k$, there exists $L>0$ such that
\begin{equation}
    \textrm{dist}(\textbf{0},\partial \Phi(x^{k+1}))\leq L\|x^{k+1}-x^k\|_2.
\end{equation}
\end{lemma}
\begin{proof}
The subdifferential of $\Phi(x^{k+1})$ can be presented as
\begin{equation}
    \partial\Phi(x^{k+1})=\left(\begin{array}{c}
                            \nabla_{y}H(y^{k+1},z^{k+1})+\partial f(y^{k+1}) \\
                             \nabla_{z}H(y^{k+1},z^{k+1})+\partial g(z^{k+1})
                          \end{array}\right).
\end{equation}
Recalling the optimization conditions of each iteration, we have
\begin{equation}
   v^{k+1}:= \left(\begin{array}{c}
                            \nabla_{y}H(y^{k+1},z^{k+1})-\nabla_{y}H(y^{k+1},z^{k})+\nabla\phi^{k}(y^{k})-\nabla\phi^{k}(y^{k+1}) \\
                             \nabla_{z}H(y^{k+1},z^{k+1})-\nabla_{z}H(y^{k+1},z^{k+1})+\nabla\psi^{k}(z^{k})-\nabla\psi^{k}(z^{k+1})
                          \end{array}\right)\in \partial\Phi(x^{k+1}).
\end{equation}
Therefore, we have that
\begin{eqnarray}
\|v^{k+1}\|_2&\leq& \|\nabla_{y}H(y^{k+1},z^{k+1})-\nabla_{y}H(y^{k+1},z^{k})+\nabla\phi^{k}(y^{k})-\nabla\phi^{k}(y^{k+1})\|_2\nonumber\\
&+&\|\nabla\psi^{k}(z^{k})-\nabla\psi^{k}(z^{k+1})\|_2\nonumber\\
&\leq&L_y\|z^{k+1}-z^{k}\|_2+L_{\phi^{k}}\|y^{k}-y^{k+1}\|_2+L_{\psi^{k}}\|z^{k}-z^{k+1}\|_2\nonumber\\
&\leq&\sqrt{2}(L_y+\max\{L_{\phi^{k}},L_{\psi^{k}}\})\|x^{k+1}-x^k\|_2.
\end{eqnarray}
Letting $L=\sqrt{2}(L_y+\max\{L_{\phi^{k}},L_{\psi^{k}}\})$, we can finish the proof.
\end{proof}
\begin{theorem}
If  both conditions (\ref{condition}) and (\ref{condition2}) hold, we have that
\begin{equation}
    \lim_{k}\textrm{dist}(\textbf{0},\partial \Phi(x^{k}))=0.
\end{equation}
\end{theorem}
\begin{proof}
Combining Lemmas \ref{limit} and \ref{relative}, we can directly obtain the result.
\end{proof}

\begin{proposition}
If $\Phi(x)$ is further a KL function and both conditions (\ref{condition}) and (\ref{condition2}) hold. Then, the sequence generated by Algorithm 1 converges to $x^*$ which is a critical point of $\Phi$.
\end{proposition}
\begin{proof}
Lemmas \ref{descend} and \ref{relative}, together with [Lemma 2.6, \cite{attouch2013convergence}], directly give the result.
\end{proof}
\section{Concluding remarks}
In this paper, we proposed a framework algorithm for alternating minimization algorithms. Thus, quite various alternating minimization algorithms have a same mathematical formulation. We also provide the convergence results of the proposed abstract algorithm.

In fact, the proposed algorithm can be extended to multi-block case.  The multi-block minimization model can be described as
 \begin{equation}\label{model2}
    \min_{x_1,x_2,\ldots,x_n} \Phi(x_1,x_2,\ldots,x_n)=H(x_1,x_2,\ldots,x_n)+\sum_{i=1}^n f_i(x_i),
\end{equation}
where $x_i\in \mathbb{R}^{n_i}$. The  multi-block Bregman alternating minimization algorithm then can be presented as
 \begin{eqnarray}\label{scheme2}
x_i^{k+1}&\in& \textrm{arg}\min_{x\in \mathbb{R}^{n_i}} H(x_1^{k+1},x_2^{k+1},\ldots,x_{i-1}^{k+1},x,x_{i+1}^k,\ldots,x_n^k)+f_i(x)+B_{\phi_i^k}(y,y^k).
\end{eqnarray}
The convergence analysis is similar to the one of the two-block case.
\section*{Acknowledgments}
We are grateful for the support from the National Natural Science Foundation of Hunan Province, China (13JJ2001), and the Science Project of National University
of Defense Technology (JC120201), and National Science Foundation of China (No.61402495).

\end{document}